\documentclass[reqno,11pt]{amsart}
\usepackage{a4wide,color,eucal,enumerate,mathrsfs}
\usepackage[normalem]{ulem}
\usepackage{amsmath,amssymb,epsfig,amsthm} %,bbm
\usepackage[latin1]{inputenc}
\usepackage{psfrag}
\numberwithin{equation}{section}

\newtheorem{thm}{Theorem}[section]
\newtheorem{lem}[thm]{Lemma}%[section]    
\newtheorem{cor}[thm]{Corollary}%[section]     
\numberwithin{equation}{section}

\theoremstyle{remark}
\newtheorem{rem}[thm]{Remark}%[section]     

\DeclareMathOperator{\diam}{diam}
\DeclareMathOperator{\dist}{dist}

\renewcommand{\S}{\mathbb{S}}
\newcommand{\D}{\mathbb{D}}
\newcommand{\R}{\mathbb{R}}
\newcommand{\Z}{\mathbb{Z}}

\renewcommand{\d}{{\mathrm d}}

\def\az{\alpha}

\def\dist{{\mathop\mathrm{\,dist\,}}}
\def\loc{{\mathop\mathrm{\,loc\,}}}

\def\ls{\lesssim}
\def\gs{\gtrsim}

\def\bint{{\ifinner\rlap{\bf\kern.35em--}
\int\else\rlap{\bf\kern.45em--}\int\fi}\ignorespaces}

\def\bbint{{\ifinner\rlap{\bf\kern.35em--}
\hspace{0.078cm}\int\else\rlap{\bf\kern.45em--}\int\fi}\ignorespaces}

\def\diam{{\mathop\mathrm{\,diam\,}}}

\def\bint{{\ifinner\rlap{\bf\kern.35em--}
\int\else\rlap{\bf\kern.45em--}\int\fi}\ignorespaces}

\begin{document}

\title[Duality of capacities and Sobolev extendability in the plane]
{Duality of capacities and Sobolev extendability in the plane}

\author{Yi Ru-Ya Zhang}

\address{ETH Z\"urich, Department of Mathematics, R\"amistrasse 101, 8092, Z\"urich, Switzerland}
\email{yizhang3@ethz.ch}
\thanks{{\bf Dedicated to the 60th birthday of Pekka Koskela}}
 \thanks{The author was partially funded by the European Research Council under the Grant Agreement No. 721675 ``Regularity and Stability in Partial Differential Equations (RSPDE)''. The author states that there is no conflict of interest.}

\date{\today}

%%%%%%%%%%%%%%%%%%%%%%%%%%%%%%%%%%%%%%%%%%%%%%%%%%%%%%%%%%%%%%%%%%%%%

\begin{abstract}
We reveal relations between the duality of capacities and the duality between Sobolev extendability of Jordan domains in the plane, and explain how to read the curve conditions involvd in the Sobolev extendability of Jordan domains via the duality of capacities. Finally  as an application, we give an alternative proof of the necessary condition for a Jordan planar domain to be $W^{1,\,q}$-extension domain when $2<q<\infty$. 
\end{abstract}

%%%%%%%%%%%%%%%%%%%%%%%%%%%%%%%%%%%%%%%%%%%%%%%%%%%%%%%%%%%%%%%%%%%%%

\maketitle

\section{Introduction}
Let $\Omega$ be a Jordan domain in the plane, that is, $\Omega$ is the bounded connected component of $\mathbb R^2 \setminus \gamma_0$ for  a Jordan curve $\gamma_0$. We say that $\Omega$ is a $W^{1,\,p}$-extension domain if it admits an extension 
operator $E \colon W^{1,p}(\Omega) \to W^{1,p}(\mathbb R^2)$ such that,  there exists a constant $C \ge 1$ so that for every 
$u \in W^{1,p}(\Omega)$ we have 
\[
||Eu||_{W^{1,p}(\R^2)} \le C||u||_{W^{1,p}(\Omega)} 
\]
and $Eu|_\Omega = u$.
Here we define the Sobolev space $W^{1,p}(\Omega)$, $1 \le p \le \infty,$ as
\[
 W^{1,p}(\Omega) = \left\{u \in L^1_{\loc}(\Omega) ~:~\nabla u \in 
L^p(\Omega,\R^2)\right\},
\]
where $\nabla u$ denotes the distributional gradient of $u$. The (semi)-norm in
$W^{1,p}(\Omega)$ that we consider here is 
$$||u||_{W^{1,p}(\Omega)} = ||\nabla u||_{L^p(\Omega)}. $$

The following result was proven in \cite{KRZ2015}.

\begin{thm}\label{sobolev ext}
Let $\Omega\subset \mathbb C$ be a Jordan domain. Then it is a $W^{1,\,p}$-extension domain if and only if $\tilde\Omega:=\mathbb C\setminus \overline \Omega$ is a $W^{1,\,q}$-extension domain for $1<p<\infty$ and $q=\frac p {p-1}$. 
\end{thm}

This duality of extendability is indeed hinted by the following duality of capacities in the plane, which originally comes from \cite{Z1967} and  can be applied to show, for instance, Uniformization theorem \cite{R2014}. Recall that for a given pair of continua $E,\,F \subset \overline \Omega\subset  { \mathbb R^2}$ and $1< p<\infty$, one defines the {\it $p$-capacity between $E$ and $F$ in $\Omega$} as
 $${\rm Cap}_p(E,\,F; \Omega)=\inf\{\|\nabla u\|^p_{ L^{p} (\Omega)}: \ u\in\Delta(E,\,F; \Omega)\},$$
 where $\Delta(E,\,F; \Omega)$ denotes the class of all $u\in W^{1,\,p}(\Omega)$ that are continuous
in $\Omega\cup E\cup F$ and satisfy $u=1$ on $E$, and $u=0$ on $F$. 
Observe that by definition  the $p$-capacity is increasing with respect to
$E,\,F,\,\Omega$. 

\begin{thm}[\cite{AM1999, R2008}]\label{pq duality}
Let $\Omega\subset \hat{\mathbb C}$ be a Jordan domain enclosed by four arcs $\gamma_1$, $\gamma_2$, $\gamma_3$ and $\gamma_4$ counterclockwise. Then we have 
\begin{equation}\label{dual capa}
\left({\rm Cap}_p(\gamma_1,\,\gamma_3;\,\Omega)\right)^{\frac 1 p}\left({\rm Cap}_q(\gamma_2,\,\gamma_4;\,\Omega)\right)^{\frac 1 q}= 1
\end{equation}
for $1<p<\infty$ and $q=\frac p {p-1}$. 
\end{thm}

To see how the duality in \eqref{sobolev ext} comes from \eqref{dual capa}, suppose $\Omega$ is a  $W^{1,\,q}$-extension Jordan domain with $2<q<\infty$ and write 
$$\partial \Omega=\bigcup_{i=1}^{4} \gamma_i$$
as in Theorem~\ref{pq duality}. By denoting $\tilde \Omega$ the complementary domain of $\Omega$, since $\Omega$ is a  $W^{1,\,q}$-extension Jordan domain, there exists a constant $C_0\ge 1$ so that
$${\rm Cap}_q(\gamma_2,\,\gamma_4;\,\tilde \Omega)\le C_0 \ {\rm Cap}_q(\gamma_2,\,\gamma_4;\,\Omega). $$
Then by applying Theorem~\ref{pq duality} to both $\Omega$ and $\tilde \Omega$ (together with a technical lemma \cite[Lemma 2.1]{KRZ2015} saying that we can always swap an unbounded domain with compact boundary to a bounded domain (and vice versa) with the same extendability), we conclude that
$${\rm Cap}_p(\gamma_1,\,\gamma_3;\,\Omega)\le C(p,\,C_0) {\rm Cap}_p(\gamma_1,\,\gamma_3;\,\tilde \Omega). $$
By the arbitrariness of $\gamma_1,\,\gamma_3\subset\partial \Omega$, we conclude the extendability of $p$-capacity functions in $\tilde \Omega$, i.e.\ those functions in $W^{1,\,p}(\tilde \Omega)$ take value $0$ and $1$ on two distinct subarcs of $\partial \Omega$, respectively. Then by the fact the exdentability of $p$-capacity functions in a planar Jordan domain implies the extendability of $W^{1,\,p}$-functions in that domain, we conclude the desired duality of Sobolev extendability.  This fact was proven in \cite{KRZ2015} via an indirect method, and also see the recent paper \cite{KPZ2019} for further information.

To be more specific, for $1<p<2$ in \cite{KRZ2015} we first show that, if every $p$-capacity function in a Jordan domain $\Omega$ is  extendable, then its complementary domain $\tilde \Omega$ is $(2-p)$-subhyperbolic, i.e.\  for $1<p<2$ and every $z_1,\,z_2\in \tilde \Omega$, there exist a constant $C_1>0$ and a curve $\gamma \subset  \tilde \Omega$
 joining $z_1$ and $z_2$  such that
 \begin{equation}\label{curve condition}
  \int_{\gamma}\dist(z,\partial \Omega)^{1-p}\,\d s(z) 
 \le  C_1 |z_1-z_2|^{2-p}.
\end{equation}
Then via this curve condition we constructed an extension operator for all functions in $W^{1,\,p}(\Omega)$.

The curve condition \eqref{curve condition} has been studied for a long time, up to notational change on the exponent. For example it can be used to characterize $Lip_{\az}$-extension domain; see e.g.\ \cite{GM1985}, \cite{la1985}. Also it appears in the characterization of $W^{1,\,q}$-extension domains with $q>2$; see for instance \cite{BK1996}, \cite{ko1998} and \cite{S2010}. 

In this paper we show the relation between \eqref{curve condition} and the duality of capacities \eqref{dual capa}. Towards this, let us introduce some terminology. 
Rcall that  the left-hand side of \eqref{curve condition} is called the  {\it $p$-subhyperbolic length of $\gamma$} for $1<p<2$, and accordingly we define the {\it $p$-subhyperbolic distance between $z_1,\,z_2$}, denoted by $d_p(z_1,\,z_2)$, via taking infimum of the $p$-subhyperbolic length among all the curves joining $z_1,\,z_2$ in $\Omega$. 
For more properties of this metric we refer to  \cite{S2010}.

For $1<p\le 2$ and $z_1,\,z_2\in \Omega$, the {\it $p$-capacity between $z_1$ and $z_2$ in $\Omega$} is defined as
\[
d_{{\rm Cap}_p}(z_1,\,z_2; \Omega)=\inf_{\gamma}\, {\rm Cap}_p(\gamma,\,\partial \Omega;\,\Omega), 
\]
where the infimum is again taken over all the curves connecting $z_1$ and $z_2$ inside $\Omega$. We remark that the triangle inequality for this metric follows naturally from the subadditivity of capacity; see e.g.\ \cite[Theorem 2 (vii), Chapter 4.7]{EG1992}.

The theorems below indicates the relation 
between the $p$-capacity metric and the  $(2-p)$-subhyperbolic metric   for $1<p<2$ in the plane. 
%\begin{thm}\label{mainthm 1}
%For $1< p<2$ and a Jordan domain $\Omega\subset  {\mathbb R^2}$ one has
%$$d_{{\rm Cap}_p}(z_1,\,z_2;\, \Omega)\sim d_{p}(z_1,\,z_2), $$
%with a constant depending only on $p$.  
%Moreover if  $\Omega$ is $J$-John (see Definition~\ref{def of John}) or satisfies \eqref{GM}, then
%$$d_{p}(z_1,\,z_2) \sim (d_{\Omega}(z_1,\,z_2))^{2-p}$$
%for any $z_1,\,z_2\in \Omega$ with a constant depend only on $J$ or $C_{GM}$ in \eqref{GM}. 
%\end{thm}

\begin{thm}
\label{thm:main_1a}
Let $1<p<2$ and $\Omega \subset \mathbb C$ be a Jordan domain and $p\in (1,2)$, and $z_1,\,z_2\in \Omega$. Then for the hyperbolic geodesic $\gamma$ joining $z_1,\,z_2$, we have
$$
{\rm Cap}_p(\gamma,\partial \Omega;\Omega) \sim d_p(z_1,\,z_2)\sim d_{{\rm Cap}_p}(z_1,z_2) \sim \int_{\gamma}\dist(z,\,\partial \Omega)^{1-p}\, dz,
$$
where the constants depend only on $p$. 
\end{thm}

We note that, in \cite{V1985} the relations between $n$-capacity metric and quasihyperbolic metric was  studied via quasiconformal mappings in $\mathbb R^n$. With the two theorems above, we are able to read both sides of \eqref{curve condition} in terms of capacities and reveal their relations. We prove it in the last section.

\begin{cor}\label{cor}
Let $\Omega$ be a Jordan $W^{1,\,q}$-extension domain in the plane with $2<q<\infty$. Then for any two points $z_1,\,z_2\in \Omega$,  there exists a curve $\gamma \subset  \Omega$
 joining $z_1$ and $z_2$  such that
\begin{equation*} 
\int_{\gamma}\dist(z,\partial \Omega)^{\frac 1 {q-1}}\,\d s(z) 
 \lesssim |z_1-z_2|^{\frac{q-2}{q-1}},
\end{equation*}
where the constant depend only on the norm of the extension operator and $q$.
\end{cor}

All the results above can be extended to the case where $\Omega$ is not Jordan but simply connected in the plane, via exhausting $\Omega$ by a sequence of Jordan domains. However, for the simplicity of the statement we omit it.

$\,$\\
{\bf Acknowledgment: }The author  would like to appreciate the anonymous referee for his careful reading and nice suggestions for organizing the paper. 

\section{Prerequisites}

We usually write the constants as positive real numbers $C(\cdot)$ with paranthesis including all the parameters on which the constant depends. The constant $C(\cdot)$ may vary between appearances, even within a chain of inequalities. By $a \sim b$ we mean that $b/C \le a \le Cb$ for some constant $C \ge 2$. 

For Euclidean spaces $\mathbb R^n$, we denote the distance of sets $A$ and $B$ by $\dist(A,B)$, and the diameter of a set $A$ by $\diam(A)$. Given an interval $I$ in $\mathbb R$, we call a continuous map $I \to X$ a \emph{path} and its image a \emph{curve}; the image of an injective map we call an \emph{arc}. We denote by $\ell(\gamma)$ the length of the curve $\gamma$. Furthermore, if $\gamma$ is an arc, then we refer to $\gamma[x,\,y]$ the subarc of $\gamma$ between points $x$ and $y$ in $\gamma$. The unit disk in $\R^2$ we denote by $\D$.

Recall that the image $\Gamma$ of an embedding $\S^1 \to   \mathbb C$ is called a \emph{Jordan curve} and, by the Jordan curve theorem, the set $  \mathbb C\setminus \Gamma$ has exactly two componets, both homeomorphic to the (open) unit disk $\mathbb D$. The bounded components of $ \mathbb C\setminus \Gamma$ are called \emph{Jordan domains}. By the Riemann mapping theorem, for each Jordan domain $\Omega$ in $\mathbb C$, there exists a conformal map $\mathbb D \to \Omega$. Moreover, given a Jordan domain $\Omega$ and a conformal map $\varphi\colon \mathbb D \to \Omega$, $\varphi$ has a homeomorphic extension $\overline{\mathbb D} \to \overline{\Omega}$ by the Caratheodory-Osgood theorem, see e.g.\;\cite{P1991}. 

Recall that for points $z_1$ and $z_2$ in $\D$, their hyperbolic distance is
\[
d_h(z_1,\,z_2)=\inf_{\az}\int_{\az} \frac {2}{1-|z|^2}\,|dz|,
\]
where the infimum is over all rectifiable curves $\az$ joining $z_1$ to $z_2$ in $\mathbb D$. The hyperbolic geodesics in $\D$ are arcs of (generalized) circles that intersect the unit circle orthogonally. 

Let $\Omega$ be a Jordan domain $\mathbb C$ with a base point $z_0$. Given $z\in \partial \Omega$ and $r>0$, we define the \emph{conformal annulus $A(z,r;\Omega)$} by
\[
A(z,r;\Omega) = \varphi\left( \{ x\in \overline{\D}\colon r/2 < |x-\varphi^{-1}(z)| < r\}\right),
\]
where $\varphi$ is the homeomorphic extension of a conformal map $\D \to \Omega$ satisfying $0 \mapsto z_0$. We supress again the role of the base point $z_0$ in the notation. Also for notational convenience we write $A(y_i,\,k)$ instead of $A(y_i,\,k;\, \Omega)$ if the domain $\Omega$ in question is clear from the context.

For the conformal annuli, we have the following comparison lemma; the proof of the analog of it in \cite[Page 645]{BKR1998} gives our version with notational changes.

\begin{lem}\label{conformal annulus}
Let $\Omega$ be a Jordan domain, $y_1,y_2\in \partial \Omega$, and let $\gamma$ be the hyperbolic geodesic in $\Omega$ joining $y_1$ and $y_2$. For each $k\in \Z_+$, let $\gamma_{i,k} = A(y_i,2^{-k};\,\Omega) \cap \gamma$. Then 
\begin{equation}
\label{eq:ca_1}
\diam(\gamma_{i,k})\sim \dist(\gamma_{i,k},\partial \Omega)\sim \ell(\gamma_{i,k})\sim \ell(\gamma_{i,\,k+1}).
\end{equation}
Furthermore, if $\alpha$ is a curve in $A(z_i,2^{-k})$ joining components of $\partial A(z_i,2^{-k})\cap \Omega$, then
\begin{equation}
\label{eq:ca_2}
\ell(\alpha)\gtrsim \ell(\gamma_{i,\,k}).
\end{equation}
The constants of comparability in \eqref{eq:ca_1} and \eqref{eq:ca_2} are independent of $\Omega$, points $y_1$ and $y_2$, the parameter $k$, and the (tacitly omitted) base point $z_0$.
\end{lem}

We can apply Lemma~\ref{conformal annulus} to show the following capacity estimate

\begin{lem}\label{annulus capacity}
With the assumptions and notation in Lemma~\ref{conformal annulus}, for every conformal annulus $A(y_i,\,k;\, \Omega)$ with $\gamma_{i,\,k}\subset \gamma[z_1,\,z_2]$,  one has
\[
\ell(\gamma_{i,\,k})^{2-p} \lesssim  \int_{A(y_i,\,k)}|\nabla u|^p\, dx
\] 
for each $u\in \Delta(\partial \Omega\cap A(y_i,\,k;\, \Omega),\,\gamma_{i,\,k};\,A(y_i,\,k;\, \Omega))$, where the constant depends only on $p$. 
\end{lem}

\begin{proof}
	Fix the domain $\Omega$ and let $q=\frac p {p-1}.$
By Theorem~\ref{pq duality}, we only need to bound the $q$-capacity of (a part of) the outer boundary and (a part of) the inner boundary of the conformal annulus $A(y_i,\,k)$  inside one of the components $A$ of $A(y_i,\,k)\setminus \gamma_{i,\,k}$  from above by a multiple of $\ell(\gamma_{i,\,k})^{2-q}$.

Let us fix one of the component $A$. Inside the topological rectangle $A$, via an argument similar to \cite[Section 3.2]{KZ2016}, we can find two hyperbolic geodesics $\gamma_1\subset A(y_i,\,k)$ and $\gamma_2\subset A(y_i,\,k)$ joining $\partial \Omega$ to $\gamma_{i,\,k}$  such that 
$$\ell(\gamma_1)\sim \ell(\gamma_2)\sim \ell(\gamma_{i,\,k}),$$
and
$$ \dist_{\Omega}(\gamma_1,\,\gamma_2)\ge c_1 \ell(\gamma_{i,\,k}),$$ 
where the constants are absolute.
%Towards this, we first divide $\varphi^{-1}(\gamma_{i,\,k})$ evenly into three subarcs $\beta_1,\,\beta_2$ and $\beta_3$, where $\beta_2$ is between the others, and $\varphi\colon \mathbb D \to \Omega$ is conformal. 
%Since $\varphi$ can be homeomorphically extended to $\partial \mathbb D$, we also divide $\partial \mathbb D \cap \varphi^{-1}(A)$ into three subarcs $\theta_1,\,\theta_2$ and $\theta_3$ with the same length, correspondingly. 
%
%
%
%By Lemma~\ref{linear map} and \eqref{eq:ca_1} we have 
%\begin{equation}\label{}
%\ell(\varphi(\beta_1))\sim \ell(\varphi(\beta_2))\sim \ell(\varphi(\beta_3))\sim \ell(\gamma_{i,\,k}).\end{equation}
%Moreover since for $j=1,\,3$
%$$1\sim {\rm Cap}(\theta_j,\,\beta_{j};\,\varphi^{-1}(A(y_i,\,k)))={\rm Cap}(\varphi(\theta_j),\,\varphi(\beta_{j});\, A(y_i,\,k)),$$
%we conclude from Lemma~\ref{}
These two curves  divide $A$ into three topological rectangles $A_1,\,A_2,\,A_3$, where $\partial A_1\cap \partial A_2=\gamma_1$ and $\partial A_3\cap \partial A_2=\gamma_2$. 
Then the following test function
$$w(x)=\max\left\{0,\, 1- \frac 1 {2c_1 \ell(\gamma_{i,\,k}) } \dist_{\Omega}(x,\,A_1)\right\}$$
implies the desired estimate; observe that
$$|\nabla w|\ls \ell(\gamma_{i,\,k})^{-1}. $$
Then the lemma follows from Theorem~\ref{pq duality} since $|A|\ls \ell(\gamma_{i,\,k})^2$. 
\end{proof}

A fundamental estimate is the $p$-capacity of a spherical annulus in the plane; see e.g.\ \cite[Page 35-37]{HKM2006} for the proof.
\begin{lem}\label{p capacity}
Let $0<r<R<\infty$. Then, for any $x\in \mathbb R^2$ and $p\in (1,\infty)\setminus \{2\}$, we have
\[
{\rm Cap}_p(B(x ,\,r),\, S^1(x ,\,R); B(x ,\,R))=C(p) (R^{\frac {p-2}{p-1}}-r^{\frac {p-2}{p-1}})^{1-p}.
\]
\end{lem}

\section{Proof of Theorem~\ref{thm:main_1a}}

\begin{proof}[Proof of Theorem~\ref{thm:main_1a}]

Let $\az\subset \Omega$ be an arbitrary curve joining $z_1$ and $z_2$. 
We  extend the hyperbolic geodesic $\gamma$ to the boundary, and denote the two end points on $\partial \Omega$ by $y_1$ and $y_2$. The notation in Lemma~\ref{conformal annulus} will be applied. 

We first show that
\begin{equation}
\label{comparable 1}
d_{{\rm Cap}_p}(z_1,z_2) \sim d_p(z_1,z_2).
\end{equation}
The discussion is divided into two cases.

\noindent {\bf Case 1 : $z_1$ and $z_2$ are in the same or neighboring conformal annulus $ A(y_i,\,2^{-k}) $ for some $i\in\{1,\,2\}$.}  Then by Lemma~\ref{conformal annulus} we have
$$\int_{\gamma} \dist(z,\,\partial \Omega)^{1-p}\, dz \sim  \ell(\gamma[z_1,\,z_2])^{2-p}\sim {\rm Cap}_p(\gamma,\,\partial \Omega;\,\Omega).$$
For an arbitrary curve $\az$, there are two sub-cases: 

\noindent {\bf Case 1.1}: 
If the curve $\az$ joining $z_1,\,z_2$ satisfies
\[
\frac 1 {8}\ell(\gamma[z_1,\,z_2])\le \dist(\az,\,\partial \Omega) \ls \ell(\gamma[z_1,\,z_2]), 
\]
meaning that $\az$ is contained in a  Whitney type set with absolute constant by Lemma~\ref{conformal annulus}, and then the standard capacity estimate Lemma~\ref{p capacity} gives
\begin{eqnarray*}
\int_{\az}\dist(z,\,\partial \Omega)^{1-p}\, dz \sim  \ell(\gamma[z_1,\,z_2])^{2-p} \gs {\rm Cap}_p(\az,\,\partial \Omega;\,\Omega),
\end{eqnarray*}
with the constant depending only on $p$. The other direction follows from a proof similar to the one of \cite[Theorem 11.7]{V1971}. 

\noindent {\bf Case 1.2}: Suppose
the curve $\az$ joining $z_1,\,z_2$ satisfies
\[
\frac 1 {8}\ell(\gamma[z_1,\,z_2])\ge \dist(\az,\,\partial \Omega) 
\] 
 By our assumption and Lemma~\ref{conformal annulus}  there is a subcurve $\az'\subset \az$ such that
\[
\ell(\gamma[z_1,\,z_2]) \ge \dist(\az',\,\partial \Omega)\ge \frac 1 {16}\ell(\gamma[z_1,\,z_2])
\]
and $\diam(\az')\sim \ell(\gamma[z_1,\,z_2])$, where the constant is absolute. Hence the same reasoning as in the  previous paragraph gives
\begin{eqnarray*}
\int_{\az'}\dist(z,\,\partial \Omega)^{1-p}\,dz\sim \ell(\gamma[z_1,\,z_2])^{2-p} \sim {\rm Cap}_p(\az',\,\partial \Omega;\,\Omega)
\end{eqnarray*}
with the constant depending on $p$. 
Since $\az'\subset \az$,
$$\int_{\az'}\dist(z,\,\partial \Omega)^{1-p}\,dz \le \int_{\az}\dist(z,\,\partial \Omega)^{1-p}\,dz,$$
and
$${\rm Cap}_p(\az',\,\partial \Omega;\,\Omega)\le {\rm Cap}_p(\az,\,\partial \Omega;\,\Omega).$$
Since we take infimum among the curves $\gamma$ in the both definitions of $p$-metric and $p$-capacity metric respectively, we conclude the theorem in the special case.

\noindent {\bf Case 2: $z_1$ and $z_2$ are not in the same or neighboring conformal annulus $\varphi(A(y_i,\,2^{-k}))$ for some $i\in\{1,\,2\}$.} We may assume that $z_1$ and $z_2$ are on the boundary of some conformal annuli by our consequence of Case 1.  Then by the subadditivity of $p$-capacity we conclude that
$$d_{{\rm Cap}_p}(z_1,\,z_2) \le  {\rm Cap}_p(\gamma,\,\partial \Omega;\,\Omega) \le \sum_{\gamma_{i,\,k}\subset \gamma[z_1,\,z_2]} {\rm Cap}_p(\gamma_{i,\,k},\,\partial \Omega;\,\Omega)$$
By  Lemma~\ref{conformal annulus} (with the notation there) and Lemma~\ref{p capacity}, on each conformal annulus we have
$$ {\rm Cap}_p(\gamma_{i,\,k},\,\partial \Omega;\,\Omega)\ls {\rm Cap}_p(\gamma_{i,\,k},\,\partial \Omega;\,\mathbb R^2) \ls  \ell(\gamma_{i,\,k})^{2-p}$$
with the constant depending only on $p$. 
By Lemma~\ref{conformal annulus} again we have
\begin{equation}\label{az}
\sum_{\gamma_{i,\,k}\subset \gamma[z_1,\,z_2]}\ell(\gamma_{i,\,k})^{2-p} \ls \sum_{\az_{i,\,k}\subset \az} \int_{\az_{i,\,k}} \dist(z,\,\partial \Omega)^{1-p}\, dz
\le \int_{\az}  \dist(z,\,\partial \Omega)^{1-p}\, dz.
\end{equation}
where $\az_{i,\,k}$ is the part of $\az$ in the corresponding annulus for some arbitrary $\az$ joining $z_1$ and $z_2$. All in all, we have
$$d_{{\rm Cap}_p}(z_1,\,z_2) \ls d_{p}(z_1,\,z_2) $$
by the arbitrariness of $\az$, where the constant depends only on $p$. 

For the other direction, let $u\in \Delta(\gamma ,\, \partial \Omega; \Omega)$ be absolutely continuous along almost every line segment. 
%Here we used the fact that when $1<p<2$, the $q$-capacity between two continua in $\mathbb R^2$ is controlled by a multiple of the distance between them to the power $2-q$, where $q=\frac p {p-1}$. 
Then by Lemma~\ref{conformal annulus} and Lemma~\ref{annulus capacity}, we have
\begin{eqnarray*}
d_{p}(z_1,\,z_2) &\le& \int_{\gamma} \dist(z,\,\partial \Omega)^{1-p}\, dz \\
&\lesssim& \sum_{\gamma_{i,\,k} \subset \gamma[z_1,\,z_2]} \ell(\gamma_{i,\,k})^{2-p} \\
&\lesssim& \sum_{i,\,k}\int_{A(y_i,\,k)}|\nabla u|^p\, dx\le  \int_{\Omega} |\nabla u|^p\, dx. 
\end{eqnarray*}
Then  by taking the infimum over $u\in \Delta(\gamma ,\, \partial \Omega; \Omega)$ we obtain the other direction. 
Thus
\eqref{comparable 1} follows
with the constant depending only on $p$. 

The rest part of the theorem follows analogously according to the calculation above: For example, by choosing $\az$ as $\gamma$ in \eqref{az} one concludes via the calculation above that
$${\rm Cap}_p(\gamma,\,\partial \Omega;\,\Omega) \ls \int_{\gamma} \dist(z,\,\partial \Omega)^{1-p}\, dz\ls \int_{\Omega} |\nabla u|^p\, dx,$$
and then obtains 
$${\rm Cap}_p(\gamma,\,\partial \Omega;\,\Omega)\ls d_{{\rm Cap}_p}(z_1,\,z_2)$$
by the arbitrariness of $u\in \Delta(\gamma ,\, \partial \Omega; \Omega)$. The other inequalities in Theorem~\ref{thm:main_1a} follow similarly. 
\end{proof}

\section{The proof of  Corollary~\ref{cor}}

\begin{proof}[Proof of Corollary~\ref{cor}]
Let us first consider the case $z_1,\,z_2\in \partial \Omega$. We claim that for these two points, there exists a curve $\gamma\subset \Omega$ such that
$$\int_{\gamma}\dist(z,\partial \Omega)^{\frac{1}{1-q}}\,\d s(z) \ls |z_1-z_2|^{\frac{q-2}{q-1}}.$$

Let $\gamma$ be the hyperbolic geodesic  joining $z_1$ and $z_2$. 
Notice that $z_1,\,z_2$ divide $\partial \Omega$ into two subarcs; namely $\Omega=\Gamma_1\cup \Gamma_2$. Then each $\Gamma_i$ with the hyperbolic geodesic $\gamma$ gives us a Jordan domain. 
Moreover, by applying an approximation argument, we generalize Theorem~\ref{pq duality} in the case where $\gamma_1=\Gamma_1$, $\gamma_3=\Gamma_2$, and $\gamma_2$ and $\gamma_4$ there are just points $z_1$ and $z_2$, respectively. 

Since $1-p = 1/(1-q)$, we have, by Theorem \ref{thm:main_1a}, (generalized) Theorem~\ref{pq duality} and the subadditivity of $p$-capacity, 
\begin{align}
\int_{\gamma}\dist(z,\partial \Omega)^{\frac{1}{1-q}}\,\d s(z) 
&= \int_{\gamma}\dist(z,\partial \Omega)^{1-p}\,\d s(z) \ls d_{{\rm Cap}_p}(z_1,z_2) \nonumber\\
&\sim  {\rm Cap}_p (\gamma,\,\partial \Omega,\; \Omega) \nonumber\\
&\lesssim {\rm Cap}_p(\gamma ,\,\Gamma_1;\, \Omega)+{\rm Cap}_p(\gamma ,\,\Gamma_2;\, \Omega)\nonumber\\
&\lesssim \left({\rm Cap}_q( z_1 ,\, z_2 ;\,\Omega)\right)^{-\frac 1 {q-1}} \label{cap}
\end{align}

For $q>2$, Morrey's inequality \cite[4.5.3]{EG1992} yields
\[
{\rm Cap}_q( z_1,\,z_2;\,\mathbb R^2)\sim |z_1-z_2|^{2-q}.
\]
Since $\Omega$ is a  $W^{1,\,q}$-extension domain, then
$${\rm Cap}_q( z_1,\,z_2;\,\Omega)\sim{\rm Cap}_q( z_1,\,z_2;\,\mathbb R^2)\sim |z_1-z_2|^{2-q},$$
with the constant depending only on the extension operator. Hence, by \eqref{cap},
\[
\int_{\gamma}\dist(z,\partial \Omega)^{\frac 1 {q-1}}\,\d s(z)  \lesssim |z_1-z_2|^{\frac{q-2}{q-1}}. 
\]
Thus the claim follows.

Now for a pair of points $x_1,\,x_2\in \Omega$, we have three cases. If 
$$2|x_1-x_2| \le \max\{\dist(x_1,\,\partial\Omega),\,\dist(x_2,\,\partial \Omega)\},$$
then the line segment joining $x_1,\,x_2$ satisfies  the desired curve condition.
On the other hand, if 
\begin{equation}\label{assum 6}
2|x_1-x_2| \ge \max\{\dist(x_1,\,\partial\Omega),\,\dist(x_2,\,\partial \Omega)\},
\end{equation}
let $z_1,\,z_2\in \partial \Omega$ such that 
$$\dist(x_1,\,z_1)=\dist(x_1,\,\partial \Omega), \qquad \dist(x_2,\,z_2)=\dist(x_2,\,\partial \Omega). $$
Then by letting $\az=[x_1,\,z_1]\cup \gamma\cup [x_2,\,z_2]$, where $[x_1,\,z_1],\,[x_2,\,z_2]$ denote the line segment joining $x_1,\,z_1$ and $x_2,\,z_2$, respectively, and $\gamma$ is the hyperbolic geodesic joining $z_1,\,z_2$. Then by our claim above with the assumption \eqref{assum 6}, we conclude that
\begin{align*}
 \int_{\az}\dist(z,\partial \Omega)^{\frac 1 {q-1}}\,\d s(z)
 \le & \left( \int_{[x_1,\,z_1]} + \int_\gamma +\int_{[x_2,\,z_2]}\right) \dist(z,\partial \Omega)^{\frac 1 {q-1}}\,\d s(z) \\
\ls &|x_1-z_1|^{\frac{q-2}{q-1}}+|z_1-z_2|^{\frac{q-2}{q-1}}+|x_2-z_2|^{\frac{q-2}{q-1}}\\
\ls & |x_1-x_2|^{\frac{q-2}{q-1}},
\end{align*}
where we used the fact that the triangle inequality and the assumption \eqref{assum 6} yield
$$|z_1-z_2|\le |x_1-z_1|+ |x_1-x_2|+|x_2-z_2|\le 5 |x_1-x_2|. $$ 
\end{proof}

\begin{rem}
When $p=q=2$, one may apply  Theorem~\ref{pq duality} to a Jordan $W^{1,\,2}$-extension domain to similarly show that, for any $z_1,\,z_2\in \Omega$, the hyperbolic curve $\Gamma$ joining them satisfies
$$\int_{\Gamma}\dist(z,\,\Omega)^{-1}\, ds(z) \left(\log\left(1+\frac{|z_1-z_2|}{\min\{\dist(z_1,\,\partial \Omega),\,\dist(z_2,\,\partial \Omega)\}}\right)\right)^{-1}\lesssim 1,$$
provided $z_1,\,z_2$ are relatively far (compared with their distances to the boundary). This is the Gehring-Osgood characterization of quasidisks \cite{GO1980}, and it is proven that a Jordan domain in the plane is a  $W^{1,\,2}$-extension domain if and only if the domain is a quasidisk; see \cite{golavo1979,gore1990,govo1981,jo1981}. Hence one can also prove the necessity of a Jordan $W^{1,\,2}$-extension domain without using any test function. 
Moreover, this indicates the relations between Gehring-Osgood condition and  \eqref{curve condition} via Theorem~\ref{pq duality}, i.e.\ the left-hand side is comparable to some capacity ($p$-capacity for \eqref{curve condition} and $2$-capacity for Gehring-Osgood condition) inside the domain, and the right-hand side is comparable to the reciprocal of its dual capacity in the Euclidean space with the correct power.
\end{rem}

\end{document}